\documentclass[12pt]{amsart}
\usepackage{latexsym, url}
\usepackage{amsthm}
\usepackage{amsmath}
\usepackage{amsfonts}
\usepackage{amssymb}
\usepackage[dvips]{graphicx}
\usepackage{xypic}
\input xy
\xyoption{all}

\newtheorem{theo}{Theorem}[section]
\newtheorem{lemma}[theo]{Lemma}

\newtheorem{propo}[theo]{Proposition}

\newtheorem{coro}[theo]{Corollary}
\newtheorem{rem}[theo]{Remark}
\newtheorem{pb}[theo]{Problem}

\newtheorem{exam}[theo]{Example}
\newtheorem{exams}[theo]{Examples}

\newcommand\Inj{\operatorname{Inj}}

\newcommand\Ext{\operatorname{Ext}}

\newcommand\Mod{\operatorname{\bf Mod}}

\newcommand\id{\operatorname{id}}

\newcommand\Proj{\operatorname{Proj}}

\newcommand\Ban{\operatorname{\bf Ban}}
\newcommand\Mono{\operatorname{Mono}}
\newcommand\Epi{\operatorname{Epi}}

\newcommand\Ab{\operatorname{\bf Ab}}

\newcommand\Pos{\operatorname{\bf Pos}}

\newcommand\Hom{\operatorname{Hom}}

\newcommand\cc{\mathcal {C}}
\newcommand\cd{\mathcal {D}}
\newcommand\cf{\mathcal {F}}

\newcommand\cg{\mathcal {G}}

\newcommand\ci{\mathcal {I}}

\newcommand\ck{\mathcal {K}}
\newcommand\cl{\mathcal {L}}

\newcommand\crr{\mathcal {R}}
\newcommand\cs{\mathcal {S}}

\newcommand\cp{\mathcal {P}}

\newcommand\cx{\mathcal {X}}
\newcommand\cy{\mathcal {Y}}

\newcommand\cz{\mathcal {Z}}

 \newbox\noforkbox \newdimen\forklinewidth
\forklinewidth=0.3pt \setbox0\hbox{$\textstyle\smile$}
\setbox1\hbox to \wd0{\hfil\vrule width \forklinewidth depth-2pt
 height 10pt \hfil}
\wd1=0 cm \setbox\noforkbox\hbox{\lower 2pt\box1\lower
2pt\box0\relax}

\date{May 23, 2022}
 
\begin{document}
\title[Fibrantly generated weak factorization systems]
{Fibrantly generated weak factorization systems}
\author[S. Cox and J. Rosick\'{y}]
{S. Cox and J. Rosick\'{y}}
\thanks{The second author is supported by the Grant Agency of the Czech Republic under the grant 22-02964S} 

\address{
\newline S. Cox
\newline
Department of Mathematics and Applied Mathematics,\newline
Virginia Commonwealth University,\newline
1015 Floyd Ave., Richmond, VA 23284,\newline
USA}
\email{scox9@vcu.edu}

\address{
\newline J. Rosick\'{y}\newline
Department of Mathematics and Statistics,\newline
Masaryk University, Faculty of Sciences,\newline
Kotl\'{a}\v{r}sk\'{a} 2, 611 37 Brno,\newline
Czech Republic}
\email{rosicky@math.muni.cz}

\begin{abstract}
We prove that, assuming Vop\v enka's principle, every small projectivity class in an accessible category is accessible.  This conclusion is not provable in ZFC alone, and in fact carries large cardinal strength. 
\end{abstract}
\maketitle

\section{Introduction}
Injectivity in locally presentable categories $\ck$ is well understood (see \cite{AR}). For instance, every small injectivity class is accessible
and accessibly embedded to $\ck$. Closely related to small injectivity classes are cofibrantly generated weak factorization systems which permeate abstract homotopy theory (see, e.g., \cite{Ri}). Much less is known about the dual concept of projectivity. Similarly, for the related concept of a fibrantly generated weak factorization system. Our main result is that, assuming Vop\v enka's principle, every small projectivity class in a locally presentable category is accessible and accessibly embedded. On the other hand, assuming V=L, free abelian groups form a small projectivity class which is not accessible.

We provide two proofs of our main result -- while one uses the concept of purity in accessible categories and is valid in every accessible category $\ck$ with pushouts, the other is more set-theoretical and applies to every accessible category. We also show that our main result is equivalent to the fact that every object of a locally presentable category (in fact of every accessible category with pushouts) is $\mu$-pure injective for some regular cardinal $\mu$. Finally, we explain
the relation of fibrantly generated weak factorization systems to cotorsion theories that are generated by sets.

All needed facts about locally presentable and accessible categories can be found in \cite{AR}.

\noindent {\bf Acknowledgement.} We are grateful to the referee for valuable comments.

\section{Fibrant generation}
Let $\ck$ be a category and $f:A \to B$, $g: C\to D$ morphisms
such that in each commutative square
$$
\xymatrix@C=4pc@R=3pc{
A \ar [r]^{u} \ar [d]_{f}& C \ar [d]^g\\
B \ar [r]_{v}& D
}
$$
there is a diagonal $d: B \to C$ with $d \cdot  f =u$ and $g\cdot
d=v$. One says that $g$ has the \textit{right lifting property}
w.r.t.\ $f$ and that $f$ has the \textit{left lifting property} w.r.t.\
$g$.  We write $f\square g$. For a class $\cx$ of morphisms of $\ck$ we put
\begin{align*}
\cx^\square & = \{ g\,|\, g \mbox{\  has the right lifting property
w.r.t. each $f\in \cx$\}\, and} \\
{}^\square\cx &= \{ f\,|\, f \mbox{\ has the left lifting property
w.r.t. each $g\in \cx$\}}.
\end{align*}
We will consider pairs $(\cl,\crr)$ of classes of morphisms on a locally presentable category $\ck$ such that $\crr=\cl^\square$ and $\cl={}^\square\crr$. We will call these pairs \textit{saturated}. A saturated pair is called \textit{cofibrantly generated} if $\crr=\cx^\square$ for a set $\cx$ and \textit{fibrantly generated} if $\cl={}^\square{\cy}$ for a set $\cy$. 

A saturated pair $(\cl,\crr)$ is called a \textit{weak factorization system} if every morphism of $\ck$ factorizes as an $\cl$-morphism followed by an $\crr$-morphism. Every cofibrantly generated saturated pair in a locally presentable category is a weak factorization system \cite{Be}.

\begin{pb} 
{
\em
Is every fibrantly generated pair $(\cl,\crr)$ a weak factorization system?
}
\end{pb}
In what follows, $0$ will denote an initial object of $\ck$ and $1$ a terminal one. An object $K$ is $\cl$-\textit{injective} if $K\to 1$ is in $\cl^\square$ and $\crr$-\textit{projective} if $0\to K$ is in ${}^\square\crr$. $\cl$-$\Inj$ or $\crr$-$\Proj$ denote the full subcategories of $\ck$ consisting of $\cl$-injectives or $\crr$-projectives, respectively.  A \emph{small projectivity class} is a class of the form $\crr$-$\Proj$ where $\crr$ is a set.  A class of morphisms $\cl$ is 
\textit{left-cancellable} if $gf\in\cl$ implies that $f\in\cl$. In a category with products, an $\cl$-\textit{cogenerator} is a set $\cg$ of objects such that, for every object $K$, the canonical morphism
$$
\gamma_K:K\to K^\ast=\prod\limits_{C\in\cg} C^{\ck(K,C)}
$$
is in $\cl$ (see \cite{AHRT1}).  

\begin{propo}\label{cogen}
Let $\cl$ be left-cancellable and $\ck$ has products and 
an $\cl$-injective $\cl$-cogenerator. Then $(\cl,\cl^\square)$
is a fibrantly generated weak factorization system. 
\end{propo}
\begin{proof}
Following \cite[Poposition 1.6]{AHRT}, $(\cl,\cl^\square)$ is a weak
factorization system. Let $\cy=\{C\to 1\,|\, C\in\cg\}$ where $\cg$
is an $\cl$-injective $\cl$-cogenerator. Clearly, $\cy\subseteq\cl^\square$, hence $\cl\subseteq {}^\square\cy$. On the other hand, let $f:K\to L$ be in ${}^\square\cy$. Since $K^\ast$ is ${}^\square\cy$-injective, there is $g:L\to K^\ast$
such that $gf=\gamma_K$. Since $\cl$ is left-cancellable, $f\in\cl$.
\end{proof}

\begin{exam}\label{poset}
{
\em
(1) Let $\cl$ consist of embeddings in the category $\Pos$ of posets.  
Then $\cl$-injectives are complete lattices and a two-element chain $2$
is an $\cl$-injective $\cl$-cogenerator. In fact, isotone maps $K\to 2$ correspond to down-sets in $K$. If $a,b$ are incomparable elements 
of $K$ then there exists a down-set $Z\subseteq K$ such that $a\in Z$ and $a\notin Z$. Hence $\gamma_K$ is an embedding. Since $\cl$ is
left-cancellable, following \ref{cogen}, $(\cl,\cl^\square)$ is a weak factorization system fibrantly generated by $2\to 1$.

(2) In every category with an injective cogenerator the class $\Mono$ of monomorphisms forms a fibrantly generated weak factorization system $(\Mono,\Mono^\square)$. This includes every Grothendieck topos and every Grothendieck abelian category. In particular, the category $R$-$\Mod$ of $R$-modules. 

(3) Let $\Ban$ be the category of Banach spaces and linear maps of norm $\leq 1$. Let $p:\Bbb C\to 1$. Following the Hahn--Banach theorem, $\cl={}^\square p$ is the class of linear isometries. $\cl$-injective Banach spaces are precisely Banach spaces $C(X)$ where $X$ is an extremally disconnected compact Hausdorff space and $\Ban$ has enough $\cl$-injectives (see \cite{Co}). In particular, $\Bbb C$ is $\cl$-injective. On the other hand if $f:A\to B$ is not an isometry, witnessed by a vector $x$ of norm $1$ (i.e. $\parallel f x\parallel < 1$), take $g: A\to\Bbb C$ of norm $1$, such that $|g x|=1$ (by the Hahn--Banach theorem). Supposing $g =hf$, we get
$$
1 = |g x| = |h f x| \leq \parallel h\parallel\cdot \parallel f x\parallel < \parallel h\parallel,
$$
which is a contradiction. Hence $\Bbb C$ is an $\cl$-injective $\cl$-cogenerator. Since $\cl$ is left cancellable, $(\cl,\cl^\square)$ is a fibrantly generated weak factorization system.
}
\end{exam}

\begin{rem} \label{cogen1}
{
\em
Whenever $\cl={}^\square \cy$ for a set $\cy$ consisting of morphisms of the form $C\to 1$, then $\cl$ is left-cancellable.
}
\end{rem}

\begin{rem}\label{cofgen}
{
\em
Let $(\cl,\crr)$ be a saturated pair in a locally presentable
category $\ck$ cofibrantly generated by $\cx\subseteq\cl$. Then $(\cl,\crr)$ is a weak factorization system and $\cl$ consists of retracts of \textit{cellular} morphisms, i.e., transfinite compositions of pushouts of morphisms from $\cx$. This is a consequence of a \textit{small object argument} (see \cite{Be}). 

One cannot expect this for fibrantly generated saturated pair $(\cl,\crr)$ in a locally presentable category. But, if $(\cl,\crr)$
is a weak factorization system fibrantly generated by $\cy$ then $\crr$ consists of retracts of transfinite cocompositions of pullbacks of elements of $\cy$; the authors of \cite{BHKKRS} call the latter $\cy$-\textit{Postnikov towers}. They show in \cite[3.1]{BHKKRS} that the model category of non-negatively graded chain complexes of vector spaces is fibrantly generated.
}
\end{rem}

A \textit{cotorsion theory} is a pair $(\cf,\cc)$ of classes of $R$-modules such that
\begin{align*}
\cc& =\cf{}^\perp=\{ C\,|\, \Ext^1 (F,C)=0 \quad\mbox{for all\quad} F\in\cf \}
\quad\mbox{and}\\
\cf & ={}^\perp\cc=\{ F\,|\,\Ext^1 (F,C)=0 \quad \mbox{for all\quad} C\in \cc\}\,.
\end{align*}
A cotorsion theory $(\cf, \cc)$ is \textit{generated by a set} if there is a set $\cx$ such
that $\cf={}^\perp\cx$, and if this holds, we say \emph{$\cx$ generates the cotorsion theory $(\cf, \cc)$}. It is called \textit{cogenerated by a set} if there is a set $\cy$ such that $\cc=\cy^\perp$, and if this holds, we say that \emph{$\cy$ cogenerates the cotorsion theory $(\cf, \cc)$}.\footnote{The definitions of ``generated by a set" and ``cogenerated by a set" are sometimes reversed in the literature.}
 
$\cf$-monomorphisms are monomorphisms whose cokernel is in $\cf$ and
$\cc$-epimorphisms are epimorphisms whose kernel is in $\cc$.

\begin{rem}\label{cotorsion} 
{
\em
For a cotorsion theory $(\cf,\cc)$, $(\cf$-$\Mono,\cc$-$\Epi)$ is a saturated pair (\cite[Proposition 3.1]{R2}). Moreover $\cc=(\cf$-$\Mono)$-$\Inj$ and $\cf=(\cc$-$\Epi)$-$\Proj$. If a cotorsion theory $(\cf,\cc)$ is cogenerated by a set then $(\cf$-$\Mono,\cc$-$\Epi)$ is a cofibrantly generated weak factorization system (\cite[Remark 3.2]{R2}).
}
\end{rem}

\begin{propo}\label{gen} 
A cotorsion theory $(\cf,\cc)$ is generated by a set if and only if  the saturated pair $(\cf$-$\Mono,\cc$-$\Epi)$ is fibrantly generated.
\end{propo}
\begin{proof}
Let $(\cf,\cc)$ be generated by a set $\cy$. Following \cite[Lemma 2.1]{ST},
there is a set $\cz_0$ of $\cy$-epimorphisms such that 
$\cf=\cz_0$-$\Proj$. Let $\cz=\cz_0\cup\{p\}$ where $p:C\to 1$ for
an injective cogenerator $C$. Following \ref{poset}(2), ${}^\square\cz\subseteq\Mono$. Following the dual of part I. of the proof of \cite[Lemma 4.4]{R3}, ${}^\square\cz=\cf$-$\Mono$.

Conversely, let $(\cf$-$\Mono,\cc$-$\Epi)$ be fibrantly generated by $\cy$. Let $\cz$ consist of kernels $C$ of epimorphisms $f:A\to B$ from $\cy$. Assume that $\Ext^1(X,C)=0$ for every $C\in\cz$. Consider the long exact sequence
$$
0\to\Hom(X,C)\to\Hom(X,A)\to\Hom(X,B)\to\Ext^1(X,C)\to\dots
$$
induced by $f$. Since $\Ext^1(X,C)=0$, $\Hom(X,f):\Hom(X,A)\to \Hom(X,B)$ is surjective. Hence $X\in\cf$. We have proved that $\cz$ generates $(\cf,\cc)$.
\end{proof}

\begin{rem}
{
\em
Let $(\cf,\cc)$ be a cotorsion theory generated by a set such that $\cf$ is either closed under pure submodules or $\cc$ consists of modules of finite injective dimension and the pair is hereditary. Then, assuming V=L, $(\cf,\cc)$ is cogenerated by a set (see \cite[Theorem 1.3 and Theorem 1.7]{ST1}). Following \ref{gen} and \ref{cotorsion}, the fibrantly generated saturated pair $(\cf$-$\Mono,\cc$-$\Epi)$ is also cofibrantly generated. Hence it is a weak factorization system.
}
\end{rem}
 
\section{Accessibility}
\begin{theo}\label{inj} 
Assuming Vop\v enka's principle, every object in a locally presentable category is $\mu$-pure injective for some regular cardinal $\mu$.
\end{theo}
\begin{proof}
Let $\ck$ be locally $\lambda$-presentable.
Suppose $K$ is not $\mu$-pure injective for any regular cardinal $\mu$. Hence, for every regular cardinal $\mu\geq\lambda$, there is a $\mu$-pure monomorphism $f_\mu:K\to L_\mu$ which does not split. Indeed, if every $\mu$-pure monomorphism $f:K\to L$ splits then $K$ is $\mu$-pure injective because, given a $\mu$-pure monomorphism $h:A\to B$ and $g:A\to K$ then a pushout
$$
		\xymatrix@=3pc{
			A \ar [r]^{h}\ar[d]_{g} & B \ar[d]^{\bar{g}}\\
			 K \ar [r]_{\bar{h}}& P
		}
		$$  
yields a $\mu$-pure monomorphism $\bar{h}$ (see \cite[Proposition 15]{AR1}). Since $\bar{h}$ splits, there is $s:P\to K$ such that $s\bar{h}=\id_K$. Then $s\bar{g}h=s\bar{h}g=g$.

If a morphism $f:K\to L$ is a $\mu$-pure monomorphism for all $\mu$ then $f$ is a split monomorphism. Thus there is a sequence of regular cardinals 
$$
\mu_0<\mu_1<\dots \mu_i<\dots
$$
indexed by ordinals and $\mu_i$-pure monomorphisms $f_i:K\to L_i$ such that $f_i$ is not $\mu_j$-pure for all $j>i$. Thus there is no morphism $h:f_i\to f_j$, $i<j$ in $K\downarrow\ck$. Indeed, having $h:L_i\to L_j$ with $hf_i=f_j$, then $f_i$ is $\mu_i$-pure (see \cite[Remarks 2.28]{AR}). It contradicts Vop\v enka's principle (see  
\cite[Lemma 6.3]{AR}).	
\end{proof}

\begin{exams}\label{ex}
{
\em
(1) Assume that there is a regular cardinal $\mu$ such that every abelian group is $\mu$-pure injective. Consider a $\mu$-pure epimorphism $f:A\to B$ in $\Ab$ and its kernel $g:C\to A$. Then $g$ is a $\mu$-pure monomorphism and, since $C$ is $\mu$-pure injective, $g$ splits. Thus $f$ splits, which contradicts \cite[Lemma 5.10]{CH}.

(2) Assume $0^\#$ does not exist and let $R$ be a ring which is not right perfect. Then the $R$-module $R^{(\omega)}$ is not $\mu$-pure injective for any regular cardinal $\mu$ (see the proof of \cite[Proposition 1.5]{ST}).}

\end{exams}

\begin{lemma}\label{inj1}
Let $\ck$ be a locally presentable category, $\mu$ be a regular cardinal, and $\cy$ be class of morphisms which are $\mu$-pure injective in $\ck^\to$. Then ${}^\square\cy$ is closed under $\mu$-pure subobjects in $\ck^\to$.
\end{lemma}
\begin{proof}
Let  $f:A\to B$ be in ${}^\square\cy$ and $(a,b): f'\to f$ be a $\mu$-pure monomorphism where $f':A'\to B'$. Consider $g:C\to D$ in $\cy$ and $(u',v'):f'\to g$. Since $g$ is $\mu$-pure injective in $\ck^\to$, there is $(u,v):f\to g$ in $\ck^\to$ such that $(u,v)\cdot (a,b)=(u',v')$.
Thus there is $t:B\to C$ such that $tf=u$ and $gt=v$.
$$
\xymatrix@=3pc{
B' \ar[r]^{b} \ar @/^2pc/ [rr]_{v'}& B \ar@{.>}[r]^{v} \ar@{.>}[dr]^{t} & D \\
A' \ar[r]_{a} \ar @/_2pc/ [rr]_{u'}\ar[u]^{f'} & A \ar@{.>}[r]_{u} \ar[u]_{f} & C \ar[u]_{g}
}
$$   Hence $tbf'=tfa=ua=u'$ and $gtb=vb=v'$. Thus $f'\in{}^\square\cy$.
\end{proof}

\begin{theo}\label{acc}
Let $(\cl,\crr)$ be a fibrantly generated saturated pair in a locally presentable category $\ck$. Then, assuming Vop\v enka's principle,  $\cl$ is an accessible and accessibly embedded subcategory of $\ck^\to$.
\end{theo}
\begin{proof}
Let $(\cl,\crr)$ be a saturated pair fibrantly generated by a set $\cy$ and assume Vop\v enka's principle. Since $\cy$ is a set, \ref{inj} implies there is a regular cardinal $\mu$ such that all members of $\cy$ are $\mu$-pure injective in $\ck^\to$.
Then \ref{inj1} implies that $\cl$ is closed under
$\mu$-pure subobjects. Hence $\cl$ is accessible and accessibly embedded to $\ck^\to$ (see \cite[Theorem 6.17]{AR}). 
\end{proof} 
 
\begin{coro}\label{acc1}
Assuming Vop\v enka's principle, every small projectivity class $\cp$ in a locally presentable category $\ck$ is accessible and accessibly embedded to $\ck$.  
\end{coro}
\begin{proof}
Let $\cp=\cs$-$\Proj$ for a set $\cs$.  Following \ref{acc}, ${}^\square\cs$ is accessible and accessibly embedded. Hence the same holds for $\cp$.
\end{proof}

\begin{rem}\label{inj2}
{
\em
(1) On the other hand, \ref{acc} follows from \ref{acc1}.
Indeed, $f\square g$ where $g:C\to D$  iff $f$ is projective to $(\id_C ,g):\id_C\to g$.

(2) Moreover, \ref{inj} follows from \ref{acc}. Indeed, let $\ck$ be locally presentable and $K$ be in $\ck$. Consider $\cl={}^\square t$ where $t:K\to 1$. Following $\ref{acc}$, there is a regular cardinal $\mu$ such that $\cl$ is $\mu$-accessible and closed under $\mu$-directed colimits in $\ck^\to$. Since $\cl$ contains split monomorphisms, it contains $\mu$-pure monomorphisms 
(see \cite[Proposition 2.30]{AR}). Thus $K$ is $\mu$-pure injective.
}
\end{rem}

\begin{exams}
{
\em
(1) Let $(\cf,\ci)$ be the cotorsion theory in $\Ab$ generated by $\Bbb Z$, i.e., $\cf$ is the class of Whitehead groups. Following \cite[Lemma 2.1]{ST} it is a small projectivity class. Assuming V=L, $\cf$ is the class of free groups and it is not accessible (see \cite[Chapter VII]{EM}).

(2) It follows from \ref{inj2}(2) and \ref{ex}(2) that \ref{acc} needs large
cardinals.  Note that V=L implies that $0^\#$ does not exist.
}
\end{exams}

\begin{theo}\label{acc2}
Let $(\cl,\crr)$ be a cofibrantly generated weak factorization system in a locally presentable category $\ck$.  Suppose $(\cl,\crr)$ is also fibrantly generated by a set $\cy$ of morphisms, such that each member of $\cy$ is $\mu$-pure injective in $\ck^\to$ for some regular cardinal $\mu$. Then, assuming the existence of a proper class of almost strongly compact cardinals, $\cl$ is accessible and accessibly embedded to $\ck^\to$.
\end{theo}
\begin{proof}
Following \cite[Corollary 3.3]{R}, $\cl$ is a full image of an accessible functor. Moreover, it is closed under $\mu$-pure subobjects
in $\ck^\to$ for some regular cardinal $\mu$ (see \ref{inj1}). The result
follows from \cite[Theorem 3.2]{BR}.
\end{proof}

\begin{rem}
{
\em
A weak factorization system $(\cl,\crr)$ in a category $\ck$ is \emph{accessible} if the factorization functor $\ck^\to\to\ck^{\to\to}$ is accessible. Every cofibrantly generated weak factoriztion system in a locally presentable category is accessible. Since \cite[Corollary 3.3]{R} is valid even for accessible weak factorization systems, in \ref{acc2}, the weak factorization system $(\cl,\crr)$  could only be accessible instead of cofibrantly generated.
}
\end{rem}

\begin{rem}\label{pushout}
{
\em
All results of this section are valid in every accessible category $\ck$ with pushouts. It suffices to do this in \ref{inj}, which follows from
\cite[Remark 2.30]{AR}.
}
\end{rem}

\section{Another proof of Theorem \ref{acc}}

We sketch our original proof of Theorem \ref{acc}.  While this version is less succinct than the proof given above, it may be more accessible to logicians.  Moreover, in contrast to \ref{pushout}, it does not need pushouts. All set-theoretic terminology used here agrees with \cite{Jech}.

 Assume Vop\v{e}nka's principle and that $\ck$ is an accessible category. By Remark \ref{inj2}(1) above and \cite[Corollary 6.10]{AR}, it suffices to show that every small projectivity class is accessible.  

By \cite[Theorem 5.35]{AR}, $\ck$ is equivalent to the category of models of some basic $L_\mu$-theory $T$ for some regular cardinal $\mu$.  From now on we will identify $\ck$ with this category of models.  By routine induction on formula complexity, it can be seen that \textbf{if} $\mathfrak{N}=(N,\in)$ is a $\Sigma_1$-elementary substructure of the universe of sets, $T$ and its signature are both elements and subsets of $\mathfrak{N}$, and $\mathfrak{N}$ happens to be closed under sequences of length $\mu$, \textbf{then} whenever $K$ is a model of $T$ and $K \in \mathfrak{N}$, it makes sense to form the restriction $K \restriction \mathfrak{N}$ (with underlying set $K \cap N$), and this restriction is also model of $T$.  Similarly, if $f: L \to K$ is a morphism in $\mathfrak{N}$ (i.e., $f: L \to K$ is a morphism in $\mathcal{K}$ and $f$, $L$, and $K$ are all elements of $\mathfrak{N}$), then its restriction $f \restriction \mathfrak{N}: L \restriction \mathfrak{N} \to K \restriction \mathfrak{N}$ is a morphism in the category.

 Suppose $\cs$ is a set of morphisms in $\ck$ and $\cp = \cs \text{-Proj}$.  By \cite[Corollary A.2]{Cox_MaxDecon}, there is an inaccessible cardinal $\kappa > \mu$ with the following property:  for every set $b$ there is an $\mathfrak{N} = (N,\in)$ such that:  
\begin{enumerate}
 \item $\mathfrak{N}$ is a $\Sigma_1$-elementary substructure of the universe of sets, $T$ and its signature are both elements and subsets of $\mathfrak{N}$, and $\mathfrak{N}$ is closed under $\mu$-sequences (hence the comments above regarding restrictions to $\mathfrak{N}$ are applicable);\footnote{The statement of \cite[Corollary A.2]{Cox_MaxDecon} doesn't include closure of $\mathfrak{N}$ under $\mu$ sequences, but the proof there easily arranges such closure.  Namely, in the proof of Corollary A.1, if $\kappa$ is chosen larger than $\mu$, and the $\lambda_0$ is then chosen to be of cofinality $>\mu$, then $H_{\lambda_0}$ is closed under $\mu$ sequences (and hence so is the $j[H_{\lambda_0}]$ from that proof, since the critical point of $j$ is larger than $\mu$).}

 \item $b \in \mathfrak{N}$ and $|\mathfrak{N}|<\kappa$;
 \item\label{item_S_restrictions} $|\bigcup \cs| < \kappa$ and $\mathfrak{N} \cap \kappa$ is transitive.  This implies that $s \restriction \mathfrak{N} = s$ for every $s \in \mathfrak{N} \cap \cs$.
 \item\label{item_ReflectP} ($\mathfrak{N}$ reflects membership in $\cp$):  For every $K \in \mathfrak{N}$:  $K \in \cp$ if and only if $K \restriction \mathfrak{N} \in \cp$.
 \item\label{item_ReflectFillin} ($\mathfrak{N}$ reflects existence of fill-ins) Whenever $g: P \to B$ and $f: A \to B$ are morphisms in $\mathfrak{N}$:  there is a fill-in for one of the following diagrams if and only if there is a fill-in for the other:  
 \[
 \xymatrix{
 & P \ar@{-->}[dl] \ar[dr]^-g & & &   & P \restriction \mathfrak{N} \ar@{-->}[dl] \ar[dr]^-{g \restriction \mathfrak{N}} &        \\
 A \ar[rr]^-f & & B & &  A \restriction \mathfrak{N} \ar[rr]^-{f \restriction \mathfrak{N}} & & B \restriction \mathfrak{N}  
 }
 \]

\end{enumerate}

We first claim that every element of $\cp$ is a $\kappa$-directed colimit of $<\kappa$-sized members of $\cp$.  This has nothing to do with projectivity classes, but simply uses that $\mathfrak{N}$ reflects membership in the class $\cp$.  Suppose $P \in \cp$, and let $\theta$ be a regular cardinal such that  $P \in H_\theta$, where $H_\theta$ denotes the collection of sets of hereditary cardinality less than $\theta$.  Let $U$ denote the set of $\mathfrak{N} \in P_\kappa(H_\theta)$ such that $P \in \mathfrak{N}$, and $\mathfrak{N}$ has the properties listed above (here $P$ is playing the role of the $b$ in the list of properties). In particular, each $\mathfrak{N} \in U$ reflects membership in $\cp$, so $P \restriction \mathfrak{N} \in \cp$ for every $\mathfrak{N} \in U$.  By the assumptions on $\kappa$, $U$ is a stationary subset of $P_\kappa(H_\theta)$.  It follows that $\{ P \restriction \mathfrak{N} \ | \ \mathfrak{N} \in U \}$ is a $\kappa$-directed collection (under inclusion) of members of $\cp$, each of size $<\kappa$, with union $P$.

Finally we show that $\cp$ is closed under $\kappa$-directed colimits.  Suppose 
\[
\cd=\left\langle \pi_{i,j}:P_i \to P_j \ | \ i \le j \in I \right\rangle
\]
is a $\kappa$-directed system of members of $\cp$.  Since $\kappa$ is inaccessible, it is sharply stronger than $\mu$, and hence $\ck$ is closed under $\kappa$-directed colimits.  So $\cd$ has a colimit in $\ck$, which will be denoted $M_\infty$.

Suppose, toward a contradiction, that $M_\infty \notin \cp = \cs\text{-Proj}$, as witnessed by some diagram
\begin{equation}
\xymatrix{
& M_\infty \ar[dr]^-g & \\
A \ar[rr]_-{f \in \cs} & & B
}
\end{equation}
for which there is no completion from $M_\infty$ into $A$.  Then there is a $\Sigma_1$-elementary substructure $\mathfrak{N}$ of the universe with the properties listed above such that $\cd$, $f$, and $g$ are elements of $\mathfrak{N}$ (here the ordered tuple $(\mathcal{D}, f, g)$ is playing the role of the $b$).  By property \ref{item_ReflectFillin}, the diagram
\begin{equation}\label{eq_RestrictN}
\xymatrix{
& M_\infty \restriction \mathfrak{N} \ar[dr]^-{g \restriction \mathfrak{N}} & \\
A \restriction \mathfrak{N} \ar[rr]_-{f \restriction \mathfrak{N} } & & B \restriction \mathfrak{N}
}
\end{equation}
has no completion from $M_\infty \restriction \mathfrak{N}$ into $A \restriction \mathfrak{N}$.  By property \ref{item_S_restrictions}, the bottom row of diagram \eqref{eq_RestrictN} is simply the map $f: A \to B$.  In summary, the diagram 
\begin{equation}\label{eq_MainDiag}
\xymatrix{
& M_\infty \restriction \mathfrak{N} \ar[dr]^-{g \restriction \mathfrak{N}} & \\
A \ar[rr]^-f & & B
}
\end{equation}
has no completion from $M_\infty \restriction \mathfrak{N}$ into $A$. 

On the other hand, since $|\mathfrak{N}|<\kappa$ and $I$ is $\kappa$-directed, there is an $i^* \in I$ above all members of $\mathfrak{N} \cap I$.  Consider the colimit map
\[
\pi_{i^*,\infty}: P_{i^*} \to M_\infty.
\]
Next we show there is a morphism
\[
e: M_\infty \restriction \mathfrak{N} \to P_{i^*}
\]
such that
\begin{equation}\label{eq_Inclusion}
\pi_{i^*,\infty} \circ e \text{ is the inclusion map from } M_\infty \restriction \mathfrak{N} \text{ into }  M_\infty.
\end{equation}

\noindent The map $e$ is defined as follows:  by elementarity of $\mathfrak{N}$, any member of $M_\infty \restriction \mathfrak{N}$ is of the form $\pi_{k,\infty}(x_k)$ for some $k \in I$ and some $x_k \in P_k$, where both $k$ and $x_k$ are \emph{elements} of $\mathfrak{N}$.  Then define $e\big( \pi_{k,\infty}(x_k) \big):= \pi_{k,i^*}(x_k)$.  It is routine to verify this does not depend on the choice of $k \in I \cap \mathfrak{N}$ or of $x_k \in P_k \cap \mathfrak{N}$;  see \cite{Cox_FiltGames} for a similar argument (using directedness of $I$ and elementarity of $\mathfrak{N}$).  And clearly \eqref{eq_Inclusion} holds, since 
\[
\pi_{i^*,\infty} \Big( e \big( \pi_{k,\infty}(x_k) \big) \Big) = \pi_{i^*,\infty} \Big( \pi_{k,i^*}(x_k) \Big) =  \pi_{k,\infty}(x_k).
\]  

To see that $e$ is a morphism, suppose $\rho < \mu$, $\boldsymbol{z}=\langle z_\xi \ | \ \xi < \rho \rangle$ is a sequence of members of $M_\infty \restriction \mathfrak{N}$, and $\dot{h}$ is a $\rho$-ary function symbol in the signature.  Note that $\dot{h} \in \mathfrak{N}$ because the signature is contained (as a subset) in $\mathfrak{N}$, and the sequence $\boldsymbol{z}$ is an element of $\mathfrak{N}$ by the $\mu$-closure of $\mathfrak{N}$.  It follows by $\Sigma_1$-elementarity of $\mathfrak{N}$ in the universe, and $\kappa$-directedness of $I$, that there is some $k^* \in I \cap \mathfrak{N}$ and some sequence $\boldsymbol{z^*} = \langle z^*_\xi \ | \ \xi < \rho \rangle$ of elements of $P_{k^*}$ such that $\pi_{k^*,\infty}(z^*_\xi) = z_\xi$ for each $\xi < \rho$, the entire sequence $\boldsymbol{z}^*$ is an element of $\mathfrak{N}$, and each $z^*_\xi$ is an element of $\mathfrak{N}$.  It follows that $e(z_\xi) = \pi_{k^*,i^*}(z^*_\xi)$ for each $\xi$, and that $h^{P_{k^*}}\big(  \boldsymbol{z^*} \big)$ is an element of $\mathfrak{N}$.  Then
\[
e\Big( h^{M_\infty}(\boldsymbol{z}) \Big) =  e \Big( \pi_{k^*,\infty} \big(  \underbrace{h^{P_{k^*}}(\boldsymbol{z^*})}_{\in \mathfrak{N} \cap P_{k^*}}\big) \Big) = \pi_{k^*,i^*} \big( h^{P_{k^*}}(\boldsymbol{z^*}) \big) = h^{P_{i^*}}\big( \langle \underbrace{\pi_{k^*,i^*}(z^*_\xi)}_{e(z_\xi)} \ | \ \xi < \rho \rangle \big),
\]
where the second equality is by the definition of $e$.


  Consider the diagram:
\begin{equation*}
\xymatrix{
  M_\infty \restriction \mathfrak{N}  \ar[r]^-e & P_{i^*} \ar[r]^-{\pi_{i^*,\infty}} & M_\infty \ar[d]^-g \\
A \ar[rr]_-f  & & B  
}
\end{equation*}
Since $f \in \cs$ and $P_{i^*} \in \cs\text{-Proj}$, there exists a morphism $\tau: P_{i^*} \to A$ such that
\[
f \circ \tau = g \circ \pi_{i^*,\infty}.
\]
Set $\tau':= \tau \circ e$, which yields the following commutative diagram:
\begin{equation}\label{eq_FinalDiag}
\xymatrix{
 M_\infty \restriction \mathfrak{N} \ar[d]_-{\tau'} \ar[r]^-e & P_{i^*}  \ar[r]^-{\pi_{i^*,\infty}} & M_\infty \ar[d]^-g \\
A \ar[rr]_-f  & & B  
}
\end{equation}

By \eqref{eq_Inclusion}, the top row of diagram \eqref{eq_FinalDiag} is just the inclusion map from $M_\infty \restriction \mathfrak{N}$ into $M_\infty$. Hence, $\tau'$ is a completion of the diagram \eqref{eq_MainDiag}, yielding a contradiction.

\end{document}